\documentclass[11pt]{article}
\usepackage{amsfonts}
\usepackage{amssymb,amsmath}
\usepackage{amsthm}
\usepackage{latexsym}
\usepackage{graphics}
\usepackage{epic}
\usepackage{epsfig}
\usepackage{psfrag}
\usepackage{bbm}
\usepackage{dsfont}
\usepackage{enumitem}
\usepackage{stmaryrd}
\usepackage{xcolor}
\usepackage{hyperref} 
\usepackage{soul} 

\numberwithin{equation}{section}
\def\PP{\mathbb{P}}

\def\RR{\mathbb{R}}
\def\NN{\mathbb{N}}

\def\EE{\mathbb{E}}
\def\11{\mathbbm{1}}

\def\E{\mathbb{E}}
\def\P{\mathbb{P}}
\def\R{\mathbb{R}}

\def\N{\mathbb{N}}

\def\d{\partial}

\def\ZZ{\mathbb{Z}}

\newtheorem{thm}{Theorem}[section]
\newtheorem{lem}[thm]{Lemma}

\newtheorem{dfn}[thm]{Definition}
\newtheorem{prop}[thm]{Proposition}

\newtheorem{hyp}{Assumption}

\theoremstyle{remark}
\newtheorem{rem}{Remark}

\begin{document}

\title{Lyapunov criteria for uniform convergence of conditional distributions of absorbed Markov processes}

\author{Nicolas Champagnat$^{1,2,3}$, Denis Villemonais$^{1,2,3}$}

\footnotetext[1]{Universit\'e de Lorraine, IECL, UMR 7502, Campus Scientifique, B.P. 70239,
  Vand{\oe}uvre-l\`es-Nancy Cedex, F-54506, France}
\footnotetext[2]{CNRS, IECL, UMR 7502,
  Vand{\oe}uvre-l\`es-Nancy, F-54506, France} \footnotetext[3]{Inria, TOSCA team,
  Villers-l\`es-Nancy, F-54600, France.\\
  E-mail: Nicolas.Champagnat@inria.fr, Denis.Villemonais@univ-lorraine.fr}

\maketitle

\begin{abstract}
  We study the uniform convergence to quasi-stationarity of
  multidimensional processes absorbed when one of the coordinates
  vanishes. Our results cover competitive or weakly cooperative
  Lotka-Volterra birth and death processes and Feller diffusions with
  competitive Lotka-Volterra interaction. To this aim, we develop an
  original non-linear Lyapunov criterion involving two functions, which
  applies to general Markov processes.
\end{abstract}



\section{Introduction}
\label{sec:intro}

We consider a Markov process $(X_t,t\geq 0)$ evolving in a state space
$E\cup\d$, where $\d\cap E=\emptyset$ and $\d$ is absorbing. A
quasi-stationary distribution for $X$ is a probability measure
$\nu_{QSD}$ on $E$ such that
$$
\PP_{\nu_{QSD}}(X_t\in\cdot\mid t<\tau_\d)=\nu_{QSD},\quad\forall t\geq 0,
$$
where $\tau_\d$ is the first hitting time of $\d$ and, for any probability measure $\nu$ on $E$, $\PP_{\nu}$ is the law of $X$ with
initial distribution $\nu$.

Our goal is to provide a computational method, taking the form of a
nonlinear \textit{Lyapunov type condition} (sometimes also referred to
as \textit{drift condition}) ensuring the existence and uniqueness of
a quasi-stationary distribution and the uniform convergence in total
variation of the law of $X_t$ given $X_t\not\in\d$ when
$t\rightarrow+\infty$ to this quasi-stationary distribution, which
means that there exist two constants $\gamma,C>0$ such that
\begin{align}
  \label{eqNexpoconvintro}
  \|\P_\mu(X_t\in\cdot\mid t<\tau_\d)-\nu_{QSD}\|_{TV}\leq C e^{-\gamma t},\quad\forall t\geq 0,
\end{align}
for all initial distribution $\mu$ on $E$, where $\|\cdot\|_{TV}$ is the usual total variation distance on the set of finite, signed
measures on $E$, defined by $\|\mu\|_{TV}=\sup_{f\in L^\infty(E),\ \|f\|_\infty\leq 1}|\mu(f)|$. We apply this result to two standard
models in ecology and evolution, called Lotka-Volterra (or logistic) birth-death or diffusion
processes~\cite{Lambert2005,champagnat-ferriere-al-06,champagnat-ferriere-al-08}, which have attracted a lot of attention in the past
and for which the question of uniform convergence toward a quasi-stationary distribution remains largely open in the
multi-dimensional case.

Practical (linear) Lyapunov type criteria for convergence to quasi-sta\-tio\-na\-ry distributions were also developed
in~\cite{ChampagnatVillemonais2017}. However, these results usually only entail non-uniform convergence with respect to the initial
condition. In the applications we consider here, the results of Sections~4 and~5 in~\cite{ChampagnatVillemonais2017} would ensure
the existence of two positive functions $\varphi_1\geq 1$ and $\varphi_2\leq 1$ on $E$ such that
\begin{align}
  \label{eqNexpoconvintro2}
  \|\P_\mu(X_t\in\cdot\mid t<\tau_\d)-\nu_{QSD}\|_{TV}\leq C e^{-\gamma t}\frac{\mu(\varphi_1)}{\mu(\varphi_2)},\quad\forall t\geq 0,
\end{align}
for all initial distribution $\mu$ on $E$. In the cases considered
below, the function $\varphi_1$ may be taken bounded, but $\varphi_2$
is usually not bounded away from zero close to the boundary of $E$,
which leads to a non-uniform convergence result.

As may be expected, the stronger convergence result~\eqref{eqNexpoconvintro} requires a finer control of the behavior of the process
near the boundary, uniformly in $E$. As a consequence, our Lyapunov criteria are more involved than those
of~\cite{ChampagnatVillemonais2017} and the techniques used here differ: in the present article, we use the control
of the derivative of the continuous-time semi-group of the process to localize the conditioned process when it starts close to the
boundary (see Proposition~\ref{prop:Dynkin-conditionnel}), while~\cite{ChampagnatVillemonais2017} makes use of the time-discretisation of the
semi-group for this purpose. Additional arguments are also required to control the behavior of the conditioned process close to
infinity. Note however that, in the end, our approach relies on the discrete time criterion of~\cite{ChampagnatVillemonais2016b} and
it is tempting to think that one may use a combination of the present approach and of the discrete time criterion
of~\cite{ChampagnatVillemonais2017} in order to prove~\eqref{eqNexpoconvintro2} with $\varphi_2=1$. Although this would lead to more
complicated criteria than the one presented below, such an approach could also apply to processes that do not come down from
infinity, such as Orstein-Uhlenbeck processes, with uniform convergence among initial distributions in bounded subsets. 
We leave this question for future research.

The uniform convergence~\eqref{eqNexpoconvintro} is also transferred to other properties of the process. Indeed, it provides
uniformity in the time to the so-called \emph{mortality/extinction plateau} (see~\cite{MV12}), uniform convergence of the process
conditioned to late survival to the so-called $Q$-process, uniform exponential ergodicity of the $Q$-process
(see~\cite{ChampagnatVillemonais2016}), and uniform convergence of $x\mapsto e^{\lambda_0 t}\PP_x(t<\tau_\d)$ to an eigenfunction
$\eta:E\rightarrow(0,+\infty)$ for the semigroup, for some positive constant $\lambda_0>0$, when
$t\rightarrow+\infty$ (see~\cite[Theorem~2.5]{ChampagnatVillemonais2016b}).

\bigskip One of the main tools of our proofs is~\cite{ChampagnatVillemonais2016b}, where we showed that the uniform exponential
convergence of conditional distributions in total variation to a unique quasi-stationary distribution is equivalent to the following
conditions: there exists a probability measure $\nu$ on $E$ such that
\begin{itemize}
\item[(A1)] there exist $t'_0,c'_1>0$ such that for all $x\in E$,
  $$
  \PP_x(X_{t'_0}\in\cdot\mid t'_0<\tau_\d)\geq c'_1\nu(\cdot);
  $$
\item[(A2)] there exists $c'_2>0$ such that for all $x\in E$ and $t\geq 0$,
  $$
  \PP_\nu(t<\tau_\d)\geq c'_2\PP_x(t<\tau_\d).
  $$
\end{itemize}
Although it has the merit of generality, this criterion appears to be hard to check in
practice~\cite{ChampagnatVillemonais2016,ChampagnatCoulibaly-PasquierEtAl2016,ChampagnatVillemonais2015,ChampagnatVillemonais2015b,ChazottesColletEtAl2017,HortonKyprianouEtAl2018}.
In particular, it lacks computational methods for verification. This is one of the purposes of the new criterion we present. It
involves two bounded nonnegative functions $V$ and $\varphi$ such that $V(x)/\varphi(x)\rightarrow+\infty$ when $x$ converges to the
boundary of $E$ or to $\infty$, satisfying
\[
  -L\varphi\leq C_1\mathbbm{1}_K
\]
for some bounded subset $K$ of $E$ and
\[
  LV+C_2\frac{V^{1+\varepsilon}}{\varphi^\varepsilon}\leq C_3\varphi
\]
for some $\varepsilon>0$ and some constants $C_1,C_2,C_3>0$, where $L$ denotes (an extension of) the infinitesimal generator of the
Markov process $X$.

We apply this criterion to Lotka-Volterra birth and death processes,
and to competitive Lotka-Volterra Feller diffusion processes.  The
quasi-stationary behavior of (extensions of) these two models have
received a lot of attention in the one-dimensional
case~\cite{DoornErik1991,SteinsaltzEvans2007,KolbSteinsaltz2012,MartinezSanMartinEtAl2014,CCLMMS09,ChampagnatVillemonais2015,Chazottes2016,HeningKolb2019}.
We focus here on the multidimensional case, where the processes evolve
on the state spaces $E\cup\d=\mathbb{Z}_+^d$ for birth and death
processes (with $\mathbb{Z}_+=\{0,1,2,\ldots\}$) and
$E\cup\d=\mathbb{R}_+^d$ for diffusion processes, with $d\geq 2$, and
where absorption corresponds to the extinction of a single
population. This means that
$\d=\mathbb{Z}^d_+\setminus\mathbb{N}^d$ and $E=\mathbb{N}^d$ (where
$\mathbb{N}=\{1,2,\ldots\}$) for multidimensional birth and death
processes and $\d=\mathbb{R}_+^d\setminus(0,+\infty)^d$ and
$E=(0,+\infty)^d$ for multidimensional diffusions. Non-uniform
exponential convergence to quasi-stationary distributions for such
processes can  be obtained
using~\cite{ChampagnatVillemonais2017},~\cite{Velleret2018}
or~\cite{Ferre2018}.

\begin{rem}
  The case where absorption corresponds to the extinction of the whole population, i.e. $\d=\{(0,\ldots,0)\}$, can be handled
  combining the results of the present paper and those known in the one-dimensional
  case~\cite{ChampagnatVillemonais2016,ChampagnatVillemonais2015} following the methods
  of~\cite[Thm.\,1.1]{cattiaux-meleard-10}. This case was also considered in~\cite{ChazottesColletEtAl2017}.
\end{rem}

A Lotka-Volterra birth and death process in dimension $d\geq 2$ is a Markov process $(X_t,t\geq 0)$ on $\mathbb{Z}_+^d$
with transition rates $q_{n,m}$ from $n=(n_1,\ldots,n_d)\in\mathbb{Z}_+^d$ to $m\neq n$ in $\mathbb{Z}_+^d$ given by
\begin{equation*}
  q_{n,m}=
  \begin{cases}
    n_i(\lambda_i+\sum_{j=1}^d\gamma_{ij}n_j) & \text{if }m=n+e_i, \text{ for some $i\in\{1,\ldots,d\}$}\\
    n_i(\mu_i+\sum_{j=1}^d c_{ij}n_j) & \text{if }m=n-e_i, \text{ for some $i\in\{1,\ldots,d\}$}\\
    0 & \text{otherwise,}
  \end{cases}
\end{equation*}
where $e_i=(0,\ldots,0,1,0,\ldots,0)$ where the 1 is at the $i$-th coordinate. Note that 
the set $\d=\mathbb{Z}^d_+\setminus\mathbb{N}^d$ is absorbing for the process. We make the usual convention that
$$
q_{n,n}:=-q_n:=-\sum_{m\neq n}q_{n,m}.
$$
From the biological point of view, the constant $\lambda_i>0$ is the birth rate per individual of type $i\in\{1,\ldots,d\}$, the
constant $\mu_i>0$ is the death rate per individual of type $i$, $c_{ij}\geq 0$ is the rate of death of an individual of type $i$ from
competition with an individual of type $j$, and $\gamma_{ij}\geq 0$ is the rate of birth of an individual of type $i$ from cooperation with
(or predation of) an individual of type $j$.
In general, a Lotka-Volterra process could be explosive if some of the $\gamma_{ij}$ are positive, but the assumptions of the next
theorem ensure that it is not the case and that the process is almost surely absorbed in finite time.

\begin{thm}
  \label{thm:main-PNM-LV}
  Consider a competitive Lotka-Volterra birth and death process $(X_t,t\geq 0)$ in $\mathbb{Z}_+^d$ as above. Assume that the matrix
  $(c_{ij}-\gamma_{ij})_{1\leq i,j\leq d}$ defines a positive operator on $\R_+^d$ in the sense that, for all
  $(x_1,\ldots,x_d)\in\R_+^d\setminus\{0\}$, $\sum_{ij}x_i(c_{ij}-\gamma_{ij})x_j>0$. Then the process has a unique quasi-stationary
  distribution $\nu_{QSD}$ and there exist constants $C,\gamma>0$ such that, for all probability measures $\mu$ on $E=\N^d$,
  \begin{equation*}
    \left\|\PP_\mu(X_t\in\cdot\mid t<\tau_\d)-\nu_{QSD}\right\|_{TV}\leq C e^{-\gamma t},\quad \forall t\geq 0.
  \end{equation*}
\end{thm}

An important difficulty here is the fact that the
absorption rate (i.e.\ the rate of jump from a state in $E$ to a state
in $\d$) is not bounded. Birth and death processes with bounded absorption rates are much easier to study, cf.\
e.g.~\cite{ChampagnatVillemonais2016b}. The
existence of a quasi-stationary distribution for this kind of
multi-dimensional birth and death processes can also be obtained using
the theory of $R$-positive matrices, as exposed
in~\cite{FerrariKestenEtAl1996}, but without the uniform
exponential convergence~\eqref{eqNexpoconvintro}.

\medskip

A competitive Lotka-Volterra Feller diffusion process in dimension $d\geq 2$ is a Markov process $(X_t,t\geq 0)$ on $\mathbb{R}_+^d$,
where $X_t=(X^1_t,\ldots,X^d_t)$, is a solution of the stochastic differential equation
\begin{equation}
  \label{eq:SDE-LV}
  dX^i_t=\sqrt{\gamma_i X^i_t}dB^i_t+X^i_t\left(r_i-\sum_{j=1}^dc_{ij}X^j_t\right)dt,\quad\forall i\in\{1,\ldots,d\},
\end{equation}
where $(B^1_t,t\geq 0),\ldots,(B^d_t,t\geq 0)$ are independent
standard Brownian motions. The Brownian terms and the linear drift
terms correspond to classical Feller diffusions, and the quadratic
drift terms correspond to Lotka-Volterra interactions between
coordinates of the process. The variances per individual $\gamma_i$
are positive numbers, and the growth rates per individual $r_i$ can be
any real number, for all $1\leq i\leq d$. The parameters
$c_{ij}$ are assumed nonnegative for all $1\leq i,j\leq d$, which
corresponds to competitive Lotka-Volterra interaction. It is well
known that, in this case, there is global strong existence and
pathwise uniqueness for the SDE~\eqref{eq:SDE-LV}, and that it is
almost surely absorbed in finite time in
$\d=\mathbb{R}_+^d\setminus(0,+\infty)^d$ if $c_{ii}>0$ for all
$i\in\{1,\ldots,d\}$ (see~\cite{cattiaux-meleard-10} and
Section~\ref{sec:SDE}).

\begin{thm}
  \label{thm:main-SDE-LV}
  Consider a competitive Lotka-Volterra Feller diffusion $(X_t,t\geq 0)$ in $\mathbb{R}_+^d$ as above. Assume that $c_{ii}>0$
  for all $i\in\{1,\ldots,d\}$. Then the process has a unique quasi-stationary distribution $\nu_{QSD}$ and there exist constants
  $C,\gamma>0$ such that, for all probability measures $\mu$ on $E=(0,\infty)^d$,
  \begin{equation}
    \label{eq:conv}
    \left\|\PP_\mu(X_t\in\cdot\mid t<\tau_\d)-\nu_{QSD}\right\|_{TV}\leq C e^{-\gamma t},\quad \forall t\geq 0.
  \end{equation}
\end{thm}

This results was previously known in dimension 2 when the constants
$c_{ij}$ and $\gamma_{ij}$ satisfy $c_{12}\gamma_1=c_{21}\gamma_2$,
which implies that the process (after some transformations) is a
Kolmogorov diffusion (i.e.\ of the form $dY_t=dW_t-\nabla V(Y_t)dt$
for some Brownian motion $W$ and some $C^2$ function $V$,
see~\cite{cattiaux-meleard-10}). Our result is valid in any dimension
and has no restriction on the coefficients. One can also expect to
extend our result to cooperative cases (e.g.\ with $c_{21}<0$ and
$c_{12}<0$ as in~\cite{cattiaux-meleard-10}) by using our abstract
Lyapunov criterion with functions combining those used to prove
Theorems~\ref{thm:main-PNM-LV} and~\ref{thm:main-SDE-LV}. Another
motivations of our study comes from~\cite{Bansaye2019}, where the
coming down from infinity of Lotka-Volterra Feller diffusions is
studied. It appears that such processes may go extinct far from
compact sets for very large initial
conditions. Theorem~\ref{thm:main-SDE-LV} proves that this does not
prevent the conditioned process to come back to compact sets fast.

\medskip Our general non-linear Lyapunov criterion, Theorem~\ref{thm:Lyapunov}, is stated in Section~\ref{sec:lyap} and proved in
Section~\ref{secNproofs}. Sections~\ref{sec:PNM-multi-dim} and~\ref{sec:SDE} are devoted to the study of (extensions of) competitive
Lotka-Volterra birth and death processes and competitive Lotka-Volterra Feller diffusions and to the proofs of
Theorems~\ref{thm:main-PNM-LV} and~\ref{thm:main-SDE-LV}.

\medskip Since the publication of the first version of this preprint
on Arxiv, the criteria presented here have been applied to other
models, including diffusion models of deadlocks in distributed
algorithms~\cite{ChampagnatSchottEtAl2018} and stochastic reaction
networks~\cite{HansenWiuf2018}, with a weakened condition for birth
and death processes.

\section{A general Lyapunov criterion for uniform exponential convergence of conditional distributions}
\label{sec:lyap}

Our general framework is inspired from~\cite{Meyn-Tweedie-III}.

\subsection{Definitions and notations}
\label{sec:def-Markov-process}

We consider a c\`adl\`ag (right continuous with left limits)
time-homogeneous strong Markov process
$(\Omega,\mathcal{F},(\mathcal{F}_t)_{t\geq 0},(X_t,t\geq
0),(\P_x)_{x\in E\cup\{\d\}})$ 
with state space $E\cup\d$, which is assumed to be a metric space with
distance function $d$, equipped with its Borel $\sigma$-fields and
such that $E$ is measurable and $E\cap\d=\emptyset$.  In what follows,
$o$ is a fixed arbitrary point in $E$ and we define the open sets
\[
O_n=\{x\in E: d(x,\d)>1/(n+1)\text{ and }d(x,o)<n+1\},
\]
for all $n\geq 1$.

We assume that $E\cap\d=\emptyset$, that $\d$ is an absorbing
set for the process and we introduce
\[
\tau_\d=\inf\{t\geq 0,\ X_t\in\d\}.
\]
We assume that $\tau_\d<\infty$ a.s.\ and that the process is regularly absorbed, in the sense that
\begin{equation}
    \label{eq:def-tau_partial}
    \tau_\d:=\lim_{n\rightarrow+\infty} T_n,\text{ where }T_n:=\inf\{t\geq 0,X_t\not\in O_n\}
\end{equation}
and that,
for all $x\in E$ and $t\geq 0$, $\P_x(t<\tau_\d)>0$.

We also assume that, for any closed set $C$, the entry time in $C$ defined by $\tau_C=\inf\{t\geq 0: X_t\in C\}$, are $({\cal
  F}_t)_{t\geq 0}$-stopping times.\footnote{One can easily adapt our proofs to cases where entry times in other sets (e.g.\ open) are strong
  Markov times for the process. 
} In particular, since
$(E\cup\d)\setminus O_n$ is closed, $T_n$ and thus $\tau_\d$ are $(\mathcal{F}_t)_{t\geq 0}$-stopping times. 





We shall make use of the following weakened notion of generator for $X$, inspired from~\cite{Meyn-Tweedie-III} and which extends the
usual weak infinitesimal generator~\cite{Dynkin1965}.
\begin{dfn}
  \label{def:generator}
  We say that a measurable function $V:E\cup\d\rightarrow\R$ belongs to the domain $\mathcal{D}(L)$ of the weakened generator $L$
  of $X$ if there exists a measurable function $W:E\rightarrow\R$ such that, for all $n\in\N$, $t\geq 0$ and $x\in E$,
  \begin{equation}
    \label{eq:def-gene}
    \E_x\int_0^{t\wedge T_n}\left|W(X_s)\right|\,ds<\infty\text{ and }\E_xV(X_{t\wedge T_n})=V(x)+\E_x\left[\int_0^{t\wedge T_n} W(X_s) ds\right],
  \end{equation}
  and we define $LV=W$ on $E$. We also define $LV(x)=0$ for all $x\in\d$.
\end{dfn}

Due to the local form of the above Dynkin formula, it is much easier to check that $V$ belongs to $\mathcal{D}(L)$ than to the usual
domain of the weak infinitesimal generator. 

We also define the set of admissible functions to which our Lyapunov criterion applies. We will say that a measurable function on $E$
or $E\cup\d$ is \emph{locally bounded} if it is bounded on $O_n$ for all $n\geq 1$.
\begin{dfn}
  \label{def:bonnes-fonctions}
  We say that a couple $(V,\varphi)$ of functions $V$ and $\varphi$ measurable from $E\cup\d$ to $\R_+$ is an \emph{admissible couple
    of functions} if
  \begin{description}
  \item[\textmd{(i)}] $V$ and $\varphi$ are bounded, identically $0$ on $\d$, positive on $E$, and $1/V$ and $1/\varphi$ are locally
    bounded on $E$.
  \item[\textmd{(ii)}] We have the convergences 
    \begin{equation}
      \label{eq:phi=o(V)}
      \lim_{n\rightarrow+\infty}\inf_{x\in E\setminus O_n} \frac{V(x)}{\varphi(x)}=+\infty
    \end{equation}
    and
    \begin{equation}
      \label{eq:weak-weak-continuity-V}
      \lim_{n\rightarrow+\infty}V(X_{T_n})=0\quad\text{a.s.}
    \end{equation}
  \item[\textmd{(iii)}] $V$ and $\varphi$ belong to the domain of the weakened generator $L$ of $X$, $LV$ is bounded from above and
    $L\varphi$ is bounded from below.
  \end{description}
\end{dfn}

The main point of this definition of admissible functions is the
following result, whose technical proof is postponed to
Section~\ref{secNproofs}.

\begin{prop}
  \label{prop:Dynkin-conditionnel}
  Assume $(V,\varphi)$ is a couple of admissible functions.  Then, for
  all $x\in E$ and $t\geq 0$,
  $$
  \frac{\E_x[LV(X_s)]}{\E_x[\varphi(X_s)]}
  -\frac{\E_x[V(X_s)]}{\E_x[\varphi(X_s)]}\,\frac{\E_x[L\varphi(X_s)]}{\E_x[\varphi(X_s)]}\in L^1([0,t])
  $$
  and
  \begin{equation}
    \label{eq:Dynkin-conditionnel}
    \frac{\E_x[V(X_t)]}{\E_x[\varphi(X_t)]}=\frac{V(x)}{\varphi(x)}+\int_0^t\left\{\frac{\E_x[LV(X_s)]}{\E_x[\varphi(X_s)]}
      -\frac{\E_x[V(X_s)]}{\E_x[\varphi(X_s)]}\,\frac{\E_x[L\varphi(X_s)]}{\E_x[\varphi(X_s)]}\right\}\,ds.
  \end{equation}
\end{prop}


\subsection{A non-linear Lyapunov criterion}
\label{sec:nonlinear-Lyapunov}

In the following assumption, we need a pair of functions, while usual
\textit{drift conditions} of Foster-Lyapunov criteria only use one
(see for instance~\cite{Meyn-Tweedie-III}). Roughly speaking, the
function $V$ is used to control return time in compact sets from
neighborhood of the boundary, while the second one, $\varphi$, is used
to control the absorption rate. One of the main difficulty when
checking the following assumption, is that $V$ needs also to be
related to the absorption probability, via the
inequality~\eqref{eq:survie<V}.

\begin{hyp}
  \label{hyp:main}
  There exist constants $k_0\in\N$, $C_1,C_2,C_3,\varepsilon>0$ such that
    \begin{equation}
      \label{eq:true-Lyap-1}
      -L\varphi\leq C_1\mathbbm{1}_{O_{k_0}}\text{ and }      LV+C_2\frac{V^{1+\varepsilon}}{\varphi^\varepsilon}\leq C_3\varphi.
    \end{equation}
  and there exist constants $r_0,p_0>0$ and $\ell_0\in\N$ such that
  \begin{equation}
    \label{eq:survie<V}
    \P_x(r_0<\tau_\d)\leq p_0 V(x),\quad\forall x\in E\setminus O_{\ell_0}.
  \end{equation}
\end{hyp}

The main role of the first part in the above assumption is to bound
the derivative computed in Proposition~\ref{prop:Dynkin-conditionnel},
which will allow us to show that the quantity
$\frac{\E_x[V(X_t)]}{\E_x[\varphi(X_t)]}$ is uniformly bounded in
$x\in E$ for $t$ large enough. The second part of this assumption is
needed to check that the boundedness of
$\frac{\E_x[V(X_t)]}{\E_x[\varphi(X_t)]}$ implies that, conditionally
on $t<\tau_\d$, $X_t\in O_n$ with high probability for $n$ large. This
implies that the problem can be localized in this set $O_n$, so that,
in order to check Conditions~(A1--A2), it is enough to assume the
following local versions of (A1--A2), which are much easier to check.



\begin{hyp}
  \label{hyp:Doeblin-local}
  There exists a probability measure $\nu$ on $E$ such that, for all
  $n\geq 1$, there exist $a_n>0$ and $\theta_n>0$ satisfying
  \begin{equation*}
    \label{eq:Doeblin-local}
    \P_x(X_{\theta_n}\in\cdot)\geq a_n\nu, \quad\text{for all $x\in O_n$.}
  \end{equation*}
  In addition, for all $n\geq 0$, there exists a constant $D_n$ such that, for all
  $t\geq 0$,
  \begin{equation}
    \label{eq:Harnack-local}
    \sup_{x\in O_n}\P_x(t<\tau_\d)\leq D_n\inf_{x\in O_n}\P_x(t<\tau_\d).
  \end{equation}
\end{hyp}

The Lyapunov criterion of Assumption~\ref{hyp:main} will be useful to
check that the conditioned process comes back quickly in bounded
subsets of $E$ from any neighborhood of the boundary. However, it does
not imply this property for initial distribution in a neighborhood of
infinity. This is the purpose of the next assumption. Since
it only concerns the unconditioned process, it may be proved using
usual drift conditions or probabilistic arguments, as we illustrate in
the two examples of Sections~\ref{sec:PNM-multi-dim} and~\ref{sec:SDE}.
\begin{hyp}
  \label{hyp:moments-expo}
  For all $\lambda>0$, there exists $n\geq 1$ such that
  \begin{equation}
    \label{eq:moments-expo}
    \sup_{x\in E}\E_x(e^{\lambda(S_n\wedge\tau_\d)})<\infty.
  \end{equation}
  where $S_n=\inf\{t\geq 0: X_t\in \overline{O_n}\}$.
\end{hyp}

\medskip
We can now state our main general result. Its proof is given in Section~\ref{secNproofs}.

\begin{thm}
  \label{thm:Lyapunov}
  Assume that the process $(X_t,t\geq 0)$ is regularly absorbed and
  that there exists a couple of admissible functions $(V,\varphi)$
  satisfying Assumption~\ref{hyp:main}. Assume also that
  Assumptions~\ref{hyp:Doeblin-local} and~\ref{hyp:moments-expo} are
  satisfied. Then the process $X$ admits a unique quasi-stationary
  distribution $\nu_{QSD}$ and there exist constants $C,\gamma>0$ such
  that for all probability measure $\mu$ on $E$,
  \begin{equation}
    \label{eq:cv-expo-main-thm}
    \|\P_\mu(X_t\in\cdot\mid t<\tau_\d)-\nu_{QSD}\|_{TV}\leq C e^{-\gamma t},\quad\forall t\geq 0.
  \end{equation}
\end{thm}

\section{Proof of the results of Section~\ref{sec:lyap}}
\label{secNproofs}

We first give the proof of Proposition~\ref{prop:Dynkin-conditionnel} in Subsection~\ref{sec:dynkin-conditionnel} and the proof of
Theorem~\ref{thm:Lyapunov} in Subsection~\ref{sec:pf-Lyapunov}. The proofs of two technical lemmas are given in
Subsections~\ref{sec:pf-lem1} and~\ref{sec:pf-lem2}.

\subsection{Proof of Proposition~\ref{prop:Dynkin-conditionnel}}
\label{sec:dynkin-conditionnel}

  In what follows, we use the classical definition of the integral of a signed function $f$
  with respect to a positive measure $\mu$ by $\mu(f_+)-\mu(f_-)\in[-\infty,+\infty]$ (where $f_+$
  and $f_-$ denote respectively the positive and negative parts of
  $f$), which is well defined as soon as at least one of the two terms is
  finite. Classical results (as Lebesgue's theorem, Fatou's lemma and
  Fubini's theorem) still hold in this case.
  
  Using the Definition~\ref{def:generator} of the weakened infinitesimal generator, we have for all $n\geq 1$
  \begin{equation}
    \label{eq:Dynk-cond-1}
    \E_x V(X_{t\wedge T_n})=V(x)+\E_x\int_0^{t\wedge T_n} LV(X_s)ds=V(x)+\E_x\int_0^t LV(X_s)\mathbbm{1}_{s<T_n}ds.
  \end{equation}
  Note that
  \begin{align*}
    \E_x V(X_{t\wedge T_n})=\E_x \left(\11_{t<T_n} V(X_t)\right)+\E_x\left(\11_{T_n\leq t} V(X_{T_n})\right).
  \end{align*}
  Using~\eqref{eq:def-tau_partial}, the Assumption~\eqref{eq:weak-weak-continuity-V} and that $V(X_{T_n})$ is uniformly bounded,
  Lebesgue's theorem implies that
  \begin{align}
    \label{eqNLebesgue1}
    \lim_{n\rightarrow+\infty} \E_x V(X_{t\wedge T_n})=\E_x \left(\11_{t<\tau_\d} V(X_t)\right)=\E_x(V(X_t)).
  \end{align}
  Now Fatou's lemma applied to the right-hand side of~\eqref{eq:Dynk-cond-1} (using that $LV$ is
  bounded from above) gives
  \begin{equation*}
    \EE_x V(X_t)\leq V(x)+\EE_x\left[\int_0^{t\wedge\tau_\d}LV(X_s)\,ds\right]= V(x)+\int_0^t\EE_x\left[LV(X_s)\right]\,ds,
  \end{equation*}
  where we have used Fubini's theorem for the last inequality.  Since
  $V\geq 0$ and $LV$ is bounded from above, we deduce that
  $\EE_x LV(X_s)\in L^1([0,t])$ 
  and that $V(X_t)\in L^1(\Omega)$. Therefore, we can actually apply
  Lebesgue's Theorem to the right-hand side of~\eqref{eq:Dynk-cond-1}
  and hence
  \begin{equation*}
    \EE_x V(X_t)=V(x)+\int_0^t\EE_x\left[LV(X_s)\right]\,ds.
  \end{equation*}
  The same argument applies to $-\varphi$ (note that~\eqref{eq:phi=o(V)} and~\eqref{eq:weak-weak-continuity-V} imply that $\lim
  \varphi(X_{T_n})=0$ a.s.):
  \begin{equation}
    \label{eq:Dynk-cond-3}
    \EE_x \varphi(X_t)=\varphi(x)+\int_0^t\EE_x\left[L\varphi(X_s)\right]\,ds.
  \end{equation}
  Therefore $\EE_x V(X_t)$ and $\EE_x \varphi(X_t)$ are continuous with respect to $t$ and (cf.\ e.g.~\cite[Lem.~VIII.2]{brezis}), for all $T>0$,
  $t\mapsto (\EE_x V(X_t),\EE_x \varphi(X_t))$ 
  belongs to the Sobolev space $W^{1,1}([0,T],\R^2)$ (the set of
  functions from $[0,T]$ to $\mathbb{R}^2$ in $L^1$ admitting a
  derivative in the sense of distributions in $L^1$).

  In particular, since $\P_x(t<\tau_\d)>0$ and hence
  $\E_x\varphi(X_t)>0$ for all $t\in[0,T]$, we deduce from the
  continuity of $t\mapsto \EE_x \varphi(X_t)$ that
  $\inf_{t\in [0,T]}\E_x\varphi(X_t)>0$. Therefore, we deduce from
  standard properties of $W^{1,1}$ functions~\cite[Cor.~VIII.9 and
  Cor.~VIII.10]{brezis} that
  $t\mapsto \EE_x V(X_t)/\EE_x \varphi(X_t)$ also belongs to $W^{1,1}([0,T],\mathbb{R})$
  and admits as derivative
  $$
  t\mapsto\frac{\EE_x LV(X_t)}{\EE_x \varphi(X_t)}-\frac{\EE_x V(X_t)\,\EE_x L\varphi(X_t)}{[\EE_x\varphi(X_t)]^2}\in L^1([0,T]).
  $$
  Hence we have proved~\eqref{eq:Dynkin-conditionnel}.

  \subsection{Proof of Theorem~\ref{thm:Lyapunov}}
\label{sec:pf-Lyapunov}

  The proof is based on two lemmas. The first one combines Proposition~\ref{prop:Dynkin-conditionnel} and
  Assumption~\eqref{eq:true-Lyap-1} to give uniform (in $x$) controls on $\frac{\E_x V(X_t)}{\E_x\varphi(X_t)}$ for $t$ large enough.
  Its proof is given in Subsection~\ref{sec:pf-lem1}.

  \begin{lem}
    \label{propNtrueLyapunov-1}
    There exists two positive constants $A$ and $B$ such that, for all $x\in E$ and all $s\geq 0$,
    \begin{align}
      \label{eq:lem1a}
      \frac{\E_x(LV(X_s))}{\E_x(\varphi(X_s))}-\frac{\E_x(V(X_s))}{\E_x(\varphi(X_s))}\frac{\E_x(L\varphi(X_s))}{\E_x(\varphi(X_s))}\leq A-B\left(\frac{\E_x(V(X_s))}{\E_x(\varphi(X_s))}\right)^{1+\varepsilon}.
    \end{align}
    In particular, there exists $t_0>0$ such that, for all $x\in E$
    and all $t\geq t_0$,
    \begin{equation}
      \label{eq:lem1b}
      \frac{\E_x V(X_t)}{\E_x\varphi(X_t)}\leq \left(\frac{2A}{B}\right)^{1/(1+\varepsilon)}+2At_0.
    \end{equation}
  \end{lem}

  The second lemma makes use of Assumption~\eqref{eq:survie<V} to deduce from Lemma~\ref{propNtrueLyapunov-1} the following
  inequality. Its proof is given in Subsection~\ref{sec:pf-lem2}

  \begin{lem}
    \label{propNtrueLyapunov-2}
    There exist $n_0\geq 1$ and a constant $D>0$ such that, for all
    $t\geq t_0$ and all $x\in E$,
    \begin{align}
      \label{eqNlastineq}
      \E_x [V(X_{t})]\leq D\,\E_x [V(X_{t})\11_{X_{t}\in O_{n_0}}].
    \end{align} 
  \end{lem}

  We prove Theorem~\ref{thm:Lyapunov} by checking that the two conditions (A1) and (A2) in the introduction are satisfied
  (cf.~\cite{ChampagnatVillemonais2016}). 
  \bigskip

  \noindent\textit{Step 1: Proof of~(A1).}\\
  We first remark that there exists $m_0\geq 0$ such that
  $\nu(O_{m_0})>0$ (where $\nu$ and $\theta_n$ below are from
  Assumption~\ref{hyp:Doeblin-local}) and hence such that, for all
  $n\geq 1$, all $x\in O_n$ and all $k\geq 1$,
  \begin{align}
    \label{eqNiterative}
    \P_x(X_{\theta_n+k\theta_{m_0}}\in\cdot)\geq a_n\,\nu(O_{m_0})\inf_{y\in
      O_{m_0}} \P_y(X_{k\theta_{m_0}}\in \cdot)\geq a_n a_{m_0}^{k}\nu(O_{m_0})^k \nu(\cdot),
  \end{align}
  where we used Markov's property and an induction procedure over $k$. Hence we can assume without loss of generality that, for all $n\geq 1$, $\theta_n\geq r_0$ (where $r_0$ is from Assumption~\ref{hyp:main}).

  As a consequence, Assumption~\eqref{eq:survie<V}, Inequality~\eqref{eqNlastineq} and Markov's property entail that,
  for all $t\geq t_0$,
\begin{align*}
  \P_x(t+\theta_{n_0}<\tau_\d)&\leq  \P_x(t+r_0<\tau_\d) \leq D\,p_0\,\E_x [V(X_{t})\11_{X_{t}\in O_{n_0}}]\\
                  &\leq D\,p_0\,\sup_{O_{n_0}} V \,\P_x(X_{t}\in O_{n_0}).
\end{align*}
On the other hand,
\[
  \P_x(X_{t+\theta_{n_0}}\in\cdot)\geq \P_x(X_t\in O_{n_0})\inf_{y\in O_{n_0}} \P_y(X_{\theta_{n_0}}\in\cdot)\geq \P_x(X_t\in O_{n_0})\,a_{n_0}\,\nu(\cdot).
\]
Setting $t'_0=t_0+\theta_{n_0}$, the two above equations imply that,
for all $x\in E$,
\begin{align*}
  \P_x(X_{t'_0}\in\cdot\mid t'_0<\tau_\d)\geq  \frac{a_{n_0}}{Dp_0\sup_{O_{n_0}} V}\,\nu(\cdot),
\end{align*}
and hence that Assumption~(A1) holds true.

\bigskip
\noindent\textit{Step 2: Proof of~(A2).}\\
Using~\eqref{eqNiterative},  for all $n\geq m_0$, all $x\in O_n$ and all $k\geq 1$, we obtain that, for all $t\in [\theta_n+k\theta_{m_0},\theta_n+(k+1)\theta_{m_0})$,
\begin{align*}
  \P_x(t+s<\tau_\d)&\geq \P_x(X_{\theta_n+k\theta_{m_0}}\in O_{m_0})\inf_{y\in O_{m_0}} \P_y(s<\tau_\d)\\
                   &\geq a_n a_{m_0}^k\nu(O_{m_0})^{k+1}\,\inf_{y\in O_{n}} \P_y(s<\tau_\d).
\end{align*}
Setting $\lambda:=-\ln(a_{m_0}\nu(O_{m_0}))/\theta_{m_0}$ and using inequality~\eqref{eq:Harnack-local} of Assumption~\ref{hyp:Doeblin-local}, we deduce that, for all $n\geq 1$ and all $s,t\geq 0$,
\begin{equation}
  \label{eq:minoration-technique}
  \inf_{x\in O_n}\P_x(t+s <\tau_\d)\geq\frac{a_n \nu(O_{m_0})}{D_n}\exp(-\lambda t)\,\sup_{x\in O_{n}}\P_x(s<\tau_\d).
\end{equation}

Now, we apply Assumption~\ref{hyp:moments-expo} for $\lambda$ as defined above: there exists $n\geq m_0$ such
that
\begin{equation}
  \label{eq:moment-expo}
  M:=\sup_{x\in E}\E_x[\exp(\lambda(S_{n-1}\wedge\tau_\d))]<\infty.
\end{equation}
Note that, since $X_t$ is c\`adl\`ag, $X_{S_{n-1}}\in \overline{O}_{n-1}\subset O_{n}$ on the event $\{S_{n-1}<\infty\}$. Hence,
using~\eqref{eq:moment-expo} and the strong Markov property at time
$S_{n-1}$ (which is a stopping time since it is the entry time in a closed set), for all $x\in E$,
\begin{align*}
  \P_x(t<\tau_\d)  & =\P_x(t<S_{n-1}\wedge\tau_\d)+\P_x(S_{n-1}\leq t<\tau_\d)\\
  &\leq Me^{-\lambda t}+\int_0^t \sup_{y\in O_{n}}\P_y(t-s<\tau_\d)\P_x(S_{n-1}\in ds).
\end{align*}
Note that, if $S_{n-1}<\infty$, then $S_{n-1}<\tau_\d$ and $S_{n-1}=S_n\wedge \tau_\d)$. Thus, for all $s\leq t$,
$\P_y(S_{n-1}\in ds)=\P_y(S_{n-1}\wedge\tau_\d\in ds,\,S_{n-1}<\infty)\leq \P_y(S_{n-1}\wedge\tau_\d\in
ds)$. Hence, using~\eqref{eq:minoration-technique} twice, we have for all $x\in E$
\begin{align*}
  \P_x(t<\tau_\d) &\leq Me^{-\lambda t}+\int_0^t \sup_{y\in O_{n}}\P_y(t-s<\tau_\d)\P_x(S_{n-1}\wedge\tau_\d\in ds)\\
  &\leq \frac{a_{n}}{D_{n} \nu(O_{m_0})} \left[M \inf_{y\in O_{n}}\P_y(t<\tau_\d) \right. \\ & \qquad \qquad\qquad\left.+\inf_{y\in O_{n}}\P_y(t<\tau_\d)\int_0^t e^{\lambda
      s}\P_x(S_{n-1}\wedge\tau_\d\in ds) \right] \\
  &\leq  \frac{2 a_{n} M}{D_{n} \nu(O_{m_0})} \inf_{y\in O_{n}}\,\P_y(t<\tau_\d)\\
  &\leq  \frac{2 a_{n} M}{D_{n} \nu(O_{m_0})\nu(O_{n})}\P_\nu(t<\tau_\d). 
\end{align*}
Since $O_n\supset O_{m_0}$, we have $\nu(O_n)\geq \nu(O_{m_0})>0$, so the last inequality implies (A2).

  \subsection{Proof of Lemma~\ref{propNtrueLyapunov-1}}
  \label{sec:pf-lem1}

  \noindent\textit{Step 1: Proof of~\eqref{eq:lem1a}.}\\
    Fix $x\in E$ and
  $s\geq 0$. On the one hand, it follows from~\eqref{eq:true-Lyap-1}
  that
  \begin{align*}
    -\E_x(L\varphi(X_s)) & \leq\frac{C_1}{\inf_{O_{k_0}}\varphi}\,\E_x(\varphi(X_s)).
  \end{align*}
  and hence
  \begin{align}
    \label{eqNfirstinequality}
    -\frac{\E_x(L\varphi(X_s))}{\E_x(\varphi(X_s))}\E_x(V(X_s))
    \leq \frac{C_1}{\inf_{O_{k_0}}\varphi}\E_x(V(X_s))
  \end{align}
  On the other hand, we deduce from~\eqref{eq:true-Lyap-1} that
  \begin{align}
    \nonumber
    \E_x(LV(X_s))&\leq C_3\E_x(\varphi(X_s))-C_2\E_x\left(\frac{V(X_s)^{1+\varepsilon}}{\varphi(X_s)^\varepsilon}\right)\\
                 &\label{eqNsecondinequality}
                   \leq C_3\E_x(\varphi(X_s))-\frac{C_2}{2}\E_x\left(\frac{V(X_s)^{1+\varepsilon}}{\varphi(X_s)^\varepsilon}\right)-\frac{C_2}{2}\frac{\E_x(V(X_s))^{1+\varepsilon}}{\E_x\left(\varphi(X_s)\right)^\varepsilon}.
  \end{align}
  where we used H\"older's inequality to deduce that
  \begin{align*}
    \E_x(V(X_s))^{1+\varepsilon}\leq \E_x\left(\frac{V(X_s)^{1+\varepsilon}}{\varphi(X_s)^\varepsilon}\right)\,\E_x\left(\varphi(X_s)\right)^\varepsilon.
  \end{align*}
  Now, because of Assumption~\eqref{eq:phi=o(V)}, there exists $m$ large enough such that, for all $y\in E\setminus O_m$,
  \[
    \frac{C_2V^\varepsilon(y)}{2\varphi^\varepsilon(y)}\geq \frac{C_1}{\inf_{O_{k_0}}\varphi}.
  \]
  Therefore, for such a value of $m$,
  \begin{multline*}
    \frac{C_2}{2}\E_x\left(\frac{V(X_s)^{1+\varepsilon}}{\varphi(X_s)^\varepsilon}\right) \\
    \begin{aligned}
      &  =\frac{C_2}{2}\E_x\left(\frac{V(X_s)^{1+\varepsilon}}{\varphi(X_s)^\varepsilon}\11_{X_s\in
        O_m}\right)+\frac{C_2}{2}\E_x\left(\frac{V(X_s)^{1+\varepsilon}}{\varphi(X_s)^\varepsilon}\11_{X_s\in E\setminus O_m}\right)
    \\ & \geq \frac{C_1}{\inf_{O_{k_0}}\varphi}\,\E_x \left(V(X_s) \11_{X_s\in E\setminus O_m}\right) \\
    & \geq \frac{C_1}{\inf_{O_{k_0}}\varphi}\,\E_x V(X_s)-\frac{C_1\sup_{O_m}V}{\inf_{O_{k_0}}\varphi\,\inf_{O_m}\varphi}\,\E_x \varphi(X_s).
    \end{aligned}
   \end{multline*}
Finally, we obtain from the last
inequality,~\eqref{eqNfirstinequality} and~\eqref{eqNsecondinequality}
that there exists two positive constants $A,B>0$ such that
\begin{align*}
  \frac{\E_x(LV(X_s))}{\E_x(\varphi(X_s))}-\frac{\E_x(V(X_s))}{\E_x(\varphi(X_s))}\frac{\E_x(L\varphi(X_s))}{\E_x(\varphi(X_s))}\leq A-B\left(\frac{\E_x(V(X_s))}{\E_x(\varphi(X_s))}\right)^{1+\varepsilon}
\end{align*}

\bigskip

  \noindent\textit{Step 2: Proof of~\eqref{eq:lem1b}.}\\
  We define $a=(2A/B)^{1/(1+\varepsilon)}$ and
  $t_0=\frac{4}{\varepsilon B a^\varepsilon}$.
  Propositions~\ref{prop:Dynkin-conditionnel} and~\eqref{eq:lem1a}
  imply that, for all $t\geq 0$ and all $x\in E$,
  \begin{align}
    \label{eqNthirdinequality}
    \frac{\E_x V(X_t)}{\E_x\varphi(X_t)} & \leq
                                         \frac{V(x)}{\varphi(x)}+At-B\int_0^t \left(\frac{\E_x V(X_s)}{\E_x\varphi(X_s)}\right)^{1+\varepsilon}\,ds.
  \end{align}
  Since $\varepsilon>0$, this implies that, for all $x\in E$, there
  exists $u_x\in[0,t_0]$ such that
  $\frac{\E_x V(X_{u_x})}{\E_x\varphi(X_{u_x})}<a$ for any $x\in
  E$. We prove this by contradiction: assume on the contrary that for all $s\in[0,t_0]$,
  $\frac{\E_x V(X_s)}{\E_x\varphi(X_s)}\geq a$. Then, for all
  $t\in[0,t_0]$,
\[
\frac{\E_x V(X_{t})}{\E_x\varphi(X_t)}\leq \frac{V(x)}{\varphi(x)}-\frac{B}{2}\int_0^{t}\left(\frac{\E_x
    V(X_s)}{\E_x\varphi(X_s)}\right)^{1+\varepsilon}\,ds.
\]
Integrating this differential inequality up to time $t_0$ entails
\begin{align*}
  \frac{\E_x V(X_{t_0})}{\E_x\varphi(X_{t_0})} & \leq\left[\left(\frac{\varphi(x)}{V(x)}\right)^\varepsilon
    +\frac{\varepsilon B t_0}{2}\right]^{-1/\varepsilon} 
  \leq\left(\frac{2}{\varepsilon B t_0}\right)^{1/\varepsilon}=\frac{a}{2},
\end{align*}
which gives a contradiction.

Hence, using Proposition~\ref{prop:Dynkin-conditionnel} and~\eqref{eq:lem1a} again, we deduce that, for all
$t\in [t_0,2t_0]$ and all $x\in E$,
\begin{align*}
  \frac{\E_x(V(X_t))}{\E_x(\varphi(X_t))} & \leq\frac{\E_x V(X_{u_x})}{\E_x\varphi(X_{u_x})}+A(t-u_x)-B\int_{u_x}^t \left(\frac{\E_x
      V(X_s)}{\E_x\varphi(X_s)}\right)^{1+\varepsilon}\,ds \\ & \leq a+A(t-u_x)\leq a+A \,2\,t_0.  
\end{align*}
Using the same argument repetitively between time $kt_0$ (instead of
$0$) and $(k+2)t_0$ (instead of $2t_0$) gives the result for all time $t\geq t_0$.

\subsection{Proof of Lemma~\ref{propNtrueLyapunov-2}}
\label{sec:pf-lem2}

  Set
  $a'=\left(\frac{2A}{B}\right)^{1/(1+\varepsilon)}+3At_0$. Equation~\eqref{eq:phi=o(V)}
  allows us to fix $n_0\geq \ell_0$ such that
\[
\inf_{y\in E\setminus O_{n_0}}\frac{V(y)}{\varphi(y)}\geq 2a'.
\]
Lemma~\ref{propNtrueLyapunov-1} implies that, for all $t\geq t_0$,
\begin{align*}
  a' & >\frac{\E_x V(X_{t})}{\E_x\varphi(X_{t})}\\
  & \geq\frac{\E_x [V(X_{t})\11_{X_{t}\in E\setminus O_{n_0}}]+\E_x
    [V(X_{t})\11_{X_{t}\in O_{n_0}}]}{\sup_{y\in E\setminus O_{n_0}}\frac{\varphi(y)}{V(y)}\E_x [V(X_{t})\11_{X_{t}\in E\setminus
      O_{n_0} }]+\sup_{y\in E}\frac{\varphi(y)}{V(y)}\E_x [V(X_{t})\11_{X_{t}\in O_{n_0}}]} \\
  & \geq\frac{\E_x [V(X_{t})\11_{X_{t}\in E\setminus O_{n_0}}]+\E_x [V(X_{t})\11_{X_{t}\in O_{n_0}}]}{\frac{1}{2a'}\E_x
    [V(X_{t})\11_{X_{t}\in E\setminus O_{n_0} }]+\sup_{y\in E}\frac{\varphi(y)}{V(y)}\E_x [V(X_{t})\11_{X_{t}\in O_{n_0}}]}.
\end{align*}
Therefore,
\begin{align*}
  \left(a' \sup_{y\in E}\frac{\varphi(y)}{V(y)}-1\right)\,\E_x [V(X_{t})\11_{X_{t}\in O_{n_0}}]\geq\frac{1}{2}\E_x [V(X_{t})\11_{X_{t}\in E\setminus O_{n_0}}].
\end{align*}
Since $a'>\E_x V(X_{t})/\E_x\varphi(X_{t})\geq 1/\sup_{y\in E} (\varphi(y)/V(y))$,
we deduce that
there exists a constant $D>0$ such that
\begin{align*}
  \E_x [V(X_{t})]\leq D\,\E_x [V(X_{t})\11_{X_{t}\in O_{n_0}}].
\end{align*}

\section{Application to multidimensional birth and death processes absorbed when one of the coordinates hits 0}
\label{sec:PNM-multi-dim}

We consider general multitype birth and death processes in continuous time, taking values in $\ZZ_+^d$ for some $d\geq 2$. Let
$(X_t,t\geq 0)$ be a Markov process on $\ZZ_+^d$ with transition rates
$$
\text{from }n=(n_1,\ldots,n_d)\text{ to }
\begin{cases}
  n+e_j & \text{with rate }n_j b_j(n),\\ 
  n-e_j & \text{with rate }n_j d_j(n) 
\end{cases}
$$
for all $1\leq j\leq d$, with $e_j=(0,\ldots,0,1,0,\ldots,0)$, where
the nonzero coordinate is the $j$-th one,
$b(n)=(b_1(n),\ldots,b_d(n))$ and $d(n)=(d_1(n),\ldots,d_d(n))$ are
functions from $\ZZ_+^d$ to $(0,+\infty)^d$.  This model represents a
density-dependent population dynamics with $d$ types of individuals
(say $d$ species), where $b_i(n)$ (resp.\ $d_i(n)$) represents the
reproduction rate (resp.\ death rate) per individuals of species $i$
when the population is in state $n$.

Note that the forms of the birth and death rates imply that, once a
coordinate $X^j_t$ of the process hits 0, it remains equal to 0.  This
corresponds to the extinction of the population of type $j$. Hence,
the set $\d:=\ZZ_+^d\setminus\NN^d$ is absorbing for the process $X$.

We define for all $k\geq 1$
\begin{align*}
  \bar{d}(k) & =\sup_{n\in\NN^d,\ |n|=k} |n|\sum_{i=1}^d \mathbbm{1}_{n_i= 1}d_i(n), \\
  \text{and}\quad \underline{d}(k) & =\inf_{n\in\NN^d,\ |n|=k}\sum_{i=1}^d n_i\left[\mathbbm{1}_{n_i\neq 1} d_i(n)-b_i(n)\right],
\end{align*}
where $|n|:=n_1+\ldots+n_d$. We shall assume
\begin{hyp}
 \label{hyp:PNM}
  There exists $\eta>0$ small enough such that, for all $k\in\NN$ large enough,
  \begin{equation}
  \label{eq:hyp-PNM-1}
\underline{d}(k)\geq \eta \bar{d}(k),
  \end{equation}
  and
  \begin{equation}
  \label{eq:hyp-PNM-2}
    \frac{\underline{d}(k)}{k^{1+\eta}}\xrightarrow[k\rightarrow+\infty]{} +\infty.  
  \end{equation}
\end{hyp}

Note that, since the set $O_n$ is finite for all $n$, it is standard to check that any function $f:\ZZ_+^d\rightarrow\RR$ is in the
domain of the weakened infinitesimal generator of $X$ and, for all $n\in\NN^d$,
\begin{equation*}
  Lf(n)=\sum_{j=1}^d[f(n+e_j)-f(n)]n_j b_j(n) +\sum_{j=1}^d[f(n-e_j)-f(n)]n_jd_j(n). 
\end{equation*}
Under Assumption~\eqref{eq:hyp-PNM-2}, setting
$W(n)=|n|$, we have
\begin{align*}
  \frac{LW(n)}{W(n)^{1+\eta}}=-\frac{\sum_{j=1}^d n_j(d_j(n)-b_j(n))}{W(n)^{1+\eta}}\leq-\frac{\underline{d}(|n|)}{|n|^{1+\eta}}\rightarrow -\infty.
\end{align*}
This classically entails that 
\begin{equation}
  \label{eq:foster-lyap}
  \sup_{n\in\N^d}\E_n(|X_1|)<+\infty.  
\end{equation}
(The argument is very similar to the one used for
Lemma~\ref{propNtrueLyapunov-1}.) In particular, the process is
non-explosive and $\tau_\d$ is finite almost surely.  Therefore, the
process $X$ is regularly absorbed, as defined in
Section~\ref{sec:def-Markov-process}.  We
can now state the main result of the section.

\begin{thm}
  \label{thm:PNM-multi-dim}
  Under Assumption~\ref{hyp:PNM}, the multi-dimensional competitive birth and death process $(X_t,t\geq 0)$ absorbed when one of its
  coordinates hits 0 admits a unique quasi-stationary distribution $\nu_{QSD}$ and there exist constants $C,\gamma>0$ such that, for
  all probability measure $\mu$ on $\NN^d$,
  \begin{equation*}
    \|\P_\mu(X_t\in\cdot\mid t<\tau_\d)-\nu_{QSD}\|_{TV}\leq C e^{-\gamma t},\quad\forall t\geq 0.
  \end{equation*}
\end{thm}

As will appear in the proof below, to check our conditions, it is sufficient to take functions $V$ and $\varphi$ of the form
$f(|n|)\mathbbm{1}_{\NN^d}(n)$ for some $f:\NN\rightarrow\RR_+$. More precisely, the first part of Condition~\eqref{eq:true-Lyap-1}
can be checked for $\varphi(n)=f(|n|)\mathbbm{1}_{\NN^d}(n)$ with decreasing $f$ and the second part for
$V(n)=g(|n|)\mathbbm{1}_{\NN^d}(n)$ with increasing $g$, to take advantage of the drift of the birth and death process towards 0 for
large $|n|$. However, let us emphasize that, although the choice $\varphi(n)=\mathbbm{1}_{\NN^d}$ (i.e.\ $f\equiv 1$) is natural, it
cannot satisfy the first part of Condition~\eqref{eq:true-Lyap-1}, since in this case
$-L\varphi(n)=\sum_{i=1}^d\mathbbm{1}_{n_i=1}d_i(n)$ (the absorption rate) is unbounded. The direct study of $\E(V(X_t)\mid
t<\tau_\d)=\EE [V(X_t)]/\EE [\mathbbm{1}_{n\in\NN^d}(X_t)]$ is possible, but only allows to recover particular cases of
Theorem~\ref{thm:PNM-multi-dim} (cf.~\cite[Section 4]{champagnat-villemonais-2016f}). Our criterion is more flexible and better adapted to
multidimensional birth and death models of interacting populations.

\medskip

It is easy to check that Assumption~\ref{hyp:PNM} is satisfied in the general Lotka-Volterra birth and death process of the
introduction. Indeed, we clearly have $\bar{d}(k)\leq C k^2$ and
\begin{align*}
  \underline{d}(k) & \geq\inf_{n\in\NN^d,\ |n|=k}\sum_{i=1}^d (d_i(n)-b_i(n))-\sup_{n\in\NN^d,\ |n|=k}\sum_{i=1}^n d_i(n)\mathbbm{1}_{n_i=1} \\
  & \geq-Ck+\inf_{n\in\NN^d,\ |n|=k}\sum_{i,j=1}^d n_i(c_{ij}-\gamma_{ij})n_j
\end{align*}
for $C=\max_{i} \mu_i+\max_i \lambda_i+\max_{i,j}c_{ij}$. Under the assumptions of Theorem~\ref{thm:main-PNM-LV}, there exists $C'>0$ such that,
for all $n\in\NN^d$,
$$
\sum_{i,j=1}^d n_i(c_{ij}-\gamma_{ij})n_j\geq C'|n|^2.
$$
This entails Assumption~\ref{hyp:PNM}.

\begin{proof}[Proof of Theorem~\ref{thm:PNM-multi-dim}]
  Using~\eqref{eq:foster-lyap} and copying the arguments
  of~\cite[Sec.\,4.1.1 and Thm.\,4.1]{ChampagnatVillemonais2016}, one
  deduces that Assumption~\ref{hyp:Doeblin-local} is satisfied with
  $\nu=\delta_{(1,\ldots,1)}$ and that
  Assumption~\ref{hyp:moments-expo} is also satisfied.

  Hence we only have to find a couple of admissible functions
  $(V,\varphi)$ satisfying Assumption~\ref{hyp:main}.
   This couple of functions is given
  for all $n\in\ZZ_+^d$ by
  \begin{equation*}
    V(n)=
    \begin{cases}
      \sum_{k=1}^{|n|}\frac{1}{k^\alpha} & \text{if }n\in\NN^d, \\
      0 & \text{if }n\in\d
    \end{cases}
  \end{equation*}
  and
  \begin{equation*}
    \varphi(n)=
    \begin{cases}
      \sum_{k=|n|+1}^{+\infty}\frac{1}{k^\beta} & \text{if }n\in\NN^d, \\
      0 & \text{if }n\in\d,
    \end{cases}
  \end{equation*}
  for appropriate choices of $\alpha,\beta>1$. Note that the two
  functions are bounded, nonnegative and positive on $\NN^d$. Note also that, since $O_n=\{n\in\mathbb{N}^d,|n|\leq n+d\}$ (taking
  $o=(1,\ldots,1)$), so
  Conditions~(i-ii) of Definition~\ref{def:bonnes-fonctions} are
  clearly satisfied (in this discrete state space case, the
  condition~\eqref{eq:weak-weak-continuity-V} is trivial since, almost surely, $T_n=\tau_\d$ for all $n$ large enough). Note also that, since
  $\inf_{n\in\NN^d}V(n)>0$, Condition~\eqref{eq:survie<V} is also
  obviously satisfied.

  Hence, we only have to check~\eqref{eq:true-Lyap-1} since this of course implies that $V$ and $\varphi$ satisfy Point~(iii) of
  Definition~\ref{def:bonnes-fonctions}. So we compute
  \begin{align*}
    L\varphi(n) & =-\sum_{i=1}^d\frac{n_ib_i(n)}{(1+|n|)^\beta}+\sum_{i=1}^d \frac{n_i\mathbbm{1}_{n_i\neq
        1}d_i(n)}{|n|^\beta}-\sum_{i=1}^d \mathbbm{1}_{n_i=1}d_i(n)\sum_{k=|n|+1}^{+\infty}\frac{1}{k^\beta} \\ &
    \geq\frac{1}{|n|^\beta}\left[\underline{d}(|n|)-\frac{\bar{d}(|n|)}{\beta-1}\right],
  \end{align*}
  where we used the fact that
  $$
  \sum_{k=x+1}^{+\infty}\frac{1}{k^\beta}\leq\int_{x}^{+\infty}\frac{dy}{y^\beta}=\frac{1}{(\beta-1)x^{\beta-1}}.
  $$
  Hence it follows from Assumption~\eqref{eq:hyp-PNM-1} that there
  exists $\beta>1$ large enough such that $L\varphi(n)\geq 0$ for all
  $|n|$ large enough. This entails the first inequality in~\eqref{eq:true-Lyap-1}.

  We fix such a value of $\beta$. Using that
  $$
  \sup_{n\in\NN^d}V(n)=\sum_{k=1}^{+\infty}\frac{1}{k^\alpha}\leq1+\int_1^\infty\frac{dx}{x^\alpha}=\frac{\alpha}{\alpha-1}
  $$
  and
  $$
  \varphi(n)\geq\int_{|n|+1}^\infty\frac{dx}{x^\beta}=\frac{(1+|n|)^{1-\beta}}{\beta-1}\geq\frac{|n|^{1-\beta}}{2(\beta-1)}
  $$
  for $|n|$ large enough, we compute for such $n$
  \begin{align*}
    LV(n)+\frac{V^{1+\varepsilon}(n)}{\varphi^\varepsilon(n)} &
    \leq \sum_{i=1}^d\frac{n_ib_i(n)}{(|n|+1)^\alpha}-\sum_{i=1}^d\frac{n_id_i(n)\11_{n_i\neq 1}}{|n|^\alpha}+C|n|^{\varepsilon(\beta-1)} \\
     & \leq -\frac{\underline{d}(|n|)}{|n|^\alpha}+C|n|^{\varepsilon(\beta-1)},
  \end{align*}
  where $C=[\alpha/(\alpha-1)]^{1+\varepsilon}[2(\beta-1)]^{\varepsilon}$. Choosing $\alpha=1+\eta/2$ and
  $\varepsilon=\eta/[2(\beta-1)]$, Assumption~\eqref{eq:hyp-PNM-2} implies that $
  LV(n)+\frac{V^{1+\varepsilon}(n)}{\varphi^\varepsilon(n)}\leq 0$ for $n\not\in O_m$ with $m$ large enough. Since $\inf_{n\in
    O_m}\varphi(n)>0$, we have the second inequality in~\eqref{eq:true-Lyap-1}.
\end{proof}

\section{Application to multidimensional Feller diffusions absorbed when one of the coordinates hits 0}
\label{sec:SDE}


We consider a general multitype Feller diffusion $(X_t,t\geq 0)$ in
$\RR_+^d$, solution to the stochastic
differential equation
\begin{equation}
  \label{eq:Feller-diff}
  dX^i_t=\sqrt{\gamma_i X^i_t}dB^i_t+X^i_t\, r_i(X_t)dt,\quad 1\leq i\leq d,
\end{equation}
where $(B^i_t,t\geq 0)$ are independent standard Brownian motions, $\gamma_i$ are positive constants and $r_i$ are measurable maps
from $\RR_+^d$ to $\RR$. From the biological point of view, $r_i(x)$ represents the growth rate per individual of species $i$ in a
population of size vector $x\in\mathbb{R}_+^d$.  We shall make the following assumption.

\begin{hyp}
  \label{hyp:Feller}
  Assume 
  that, for all $i\in\{1,\ldots,d\}$, $r_i$ is locally H\"older on $\RR_+^d$ and
   that there
  exist $a>0$ and $0<\eta<1$ such that
  \begin{align}
    r_i(x)\leq a^\eta-x_i^\eta,
    \label{eq:hyp-Feller-1} 
  \end{align}
  and there exist constants $B_a>a$, $C_a>0$ and $D_a>0$ such that
  \begin{equation}
    \label{eq:hyp-Feller-3}
    \sum_{i=1}^d \11_{x_i\geq B_a} r_i(x)\leq C_a\left(\sum_{i=1}^d\11_{x_i\leq a} r_i(x)+D_a\right),\quad\forall x\in\RR_+^d.
  \end{equation}
\end{hyp}

This assumption implies in particular the non-explosion, strong
existence and pathwise uniqueness for~\eqref{eq:Feller-diff}. Indeed,
since $r_i$ is locally H\"older, standard arguments entail the strong
existence and pathwise uniqueness for~\eqref{eq:Feller-diff} up to the
explosion time. Now, Assumption~\eqref{eq:hyp-Feller-1} and standard
comparison results for one-dimensional diffusion processes (see
e.g. Theorem~1.1 in~\cite[Chapter~VI]{IkedaWatanabe1989}) entail that each coordinate of the process can be upper
bounded by the solution of the one-dimensional Feller diffusion
\begin{equation}
  \label{eq:upper-SDE-2}
  d\bar{X}^i_t=\sqrt{\gamma_i\bar{X}^i_t}dB^i_t+\bar{X}^i_t\left[a^\eta-\left(\bar{X}^i_t\right)^{\eta}\right]dt,\quad 1\leq i\leq d,
\end{equation}
with initial value $\bar{X}^i_0=X^i_0$. Since $\bar{X}^i$ is a diffusion on $\RR_+$ for which $+\infty$ is an entrance boundary and
$0$ an exit boundary, we deduce that each coordinate is non-explosive and hence that the unique solution to~\eqref{eq:Feller-diff} is
non-explosive. Moreover, the subset $\d=\mathbb{R}_+^d\setminus(0,+\infty)^d$ is an absorbing boundary for the diffusion
process~\eqref{eq:Feller-diff}, so the process $X_t$ is regularly absorbed in the sense of Section~\ref{sec:nonlinear-Lyapunov}.


Strong existence and pathwise uniqueness imply well-posedness of the
martingale problem, hence the strong Markov property hold on the
canonical space with respect to the natural filtration (see
e.g.~\cite{rogers-williams-00b}). Since the paths of $X$ are
continuous, the hitting times of closed subsets of $\R_+^d$ are
stopping times for this filtration. In addition, it follows
from It\^o's formula and the local boundedness of the coefficients of
the SDE that any measurable function $f:\RR_+^d\rightarrow\RR$ twice
continuously differentiable on $(0,+\infty)$ belongs to the domain of
the weakened generator of $X$ and
\[
Lf(x)=\sum_{i=1}^d\frac{\gamma_i x_i}{2}\frac{\partial^2 f}{\partial x_i^2}(x)+\sum_{i=1}^dx_ir_i(x)\frac{\partial f}{\partial x_i}(x),\quad\forall x\in(0,+\infty)^d.
\]
We can now state the main result of the section.

\begin{thm}
  \label{thm:FellerMain}
  Under Assumption~\ref{hyp:Feller}, the multi-dimensional Feller diffusion process $(X_t,t\geq 0)$ absorbed when one of its
  coordinates hits 0 admits a unique quasi-stationary distribution $\nu_{QSD}$ and there exist constants $C,\gamma>0$ such that, for
  all probability measure $\mu$ on $\NN^d$,
  \begin{equation*}
    \|\P_\mu(X_t\in\cdot\mid t<\tau_\d)-\nu_{QSD}\|_{TV}\leq C e^{-\gamma t},\quad\forall t\geq 0.
  \end{equation*}
\end{thm}

It is straightforward to check that Assumption~\ref{hyp:Feller} is satisfied in the competitive Lotka-Volterra case, that is when
$r_i(x)=r_i-\sum_{i=1}^dc_{ij}x_j$ with $c_{ij}\geq 0$ and $c_{ii}>0$ for all $1\leq i,j\leq d$. Hence Theorem~\ref{thm:main-SDE-LV}
is an immediate corollary of Theorem~\ref{thm:FellerMain}. Assumption~\ref{hyp:Feller} allows for other biologically relevant models. For
instance, one can consider ecosystems where the competition among individuals only acts when the population size reaches a level
$K>0$, which leads for instance to the SDE
\begin{align*}
dX^i_t=\sqrt{\gamma_i X^i_t}\,dB^i_t+X^i_t\left(r-\sum_{j=1}^d c_{ij}\, (X^j_t-K)_+\right)\,dt,\ X^i_0>0,
\end{align*}
where $c_{ij}$ are non-negative constants and $c_{ii}>0$ for all $i,j\in\{1,\ldots,d\}$. Note also that a similar approach (\textit{i.e.} using Theorem~\ref{thm:Lyapunov} with Lyapunov functions of the form $\prod_{i=1}^dh(x_i)$) can also be used to handle diffusion processes evolving in bounded boxes.

\medskip

Before giving the proof of Theorem~\ref{thm:FellerMain}, let us give some intuition about the choice of functions $V$ and $\varphi$.
The behavior of the process when one of the coordinates is close to 0 is similar to the behavior of a one-dimensional diffusion
absorbed at 0 started close to 0. This suggests to look for Lyapunov functions of the form $V(x)=\prod_{i=1}^d f(x_i)$ and
$\varphi(x)=\prod_{i=1}^d g(x_i)$ where the functions $f$ and $g$ satisfy our criterion for one-dimensional diffusions
like~\eqref{eq:upper-SDE-2}, say
\[
d Y_t=\sqrt{2 Y_t}dB_t+Y_t(r-Y_t^\eta)dt,
\]
with generator $A f(x)=x f''(x)+x(r-x^\eta)f'(x)$. In the one dimensional case, Conditions~\eqref{eq:true-Lyap-1} are restrictive
only close to 0 and close to $+\infty$. When $x\rightarrow 0$, assuming $f(x)=x^{\alpha}$ and $g(x)=x^{\beta}$ when $x$ is close to
0, we obtain
\[
A f(x)\sim \alpha(\alpha-1) x^{\alpha-1}\quad\text{and}\quad A g(x)\sim \beta(\beta-1) x^{\beta-1}.
\]
In particular, to satisfy the first part of~\eqref{eq:true-Lyap-1} close to 0, we need $\beta>1$, and to satisfy the second part
of~\eqref{eq:true-Lyap-1}, we need that
\[
\alpha(\alpha-1) x^{\alpha-1}+x^{\alpha-\varepsilon(\beta-\alpha)}\leq x^\alpha
\]
in the neighborhood of 0. This holds true if $0<\alpha<1$ and $\varepsilon(\beta-\alpha)<1$. Similarly, assuming $f(x)=a-x^{-\gamma}$
and $g(x)=x^{-\delta}$ close to $+\infty$, with $a,\gamma,\delta>0$, we obtain
\[
A f(x)\sim -\gamma x^{\gamma+\eta}\quad\text{and}\quad A g(x)\sim \delta x^{\eta-\delta}.
\]
For such functions $f$ and $g$ close to $+\infty$, the first part of~\eqref{eq:true-Lyap-1} is always satisfied and the second part
of~\eqref{eq:true-Lyap-1} requires
\[
-\gamma x^{\eta-\gamma}+a^{1+\varepsilon}x^{\delta\varepsilon}\leq x^{-\delta},
\]
when $x\rightarrow+\infty$. This holds true if $\eta-\gamma>\delta\varepsilon$.

This gives the intuition to check our criterion in dimension 1\footnote{Note that more efficient criteria exist for one-dimensional
  diffusions (see~\cite{ChampagnatVillemonais2015}). The interest of our criterion comes from the fact that it applies to 
  multi-dimensional processes.}. The multi-dimensional case of Theorem~\ref{thm:FellerMain} requires a more careful study, since we
need to consider cases where some of the coordinates of the process are close to 0 or $+\infty$, and the others belong to some compact
set.

\begin{proof}[Proof of Theorem~\ref{thm:FellerMain}]
  Up to a linear scaling of the coordinates, we can assume without loss of generality that $\gamma_i=2$ for all $1\leq i\leq d$, so
  we will only consider this case from now on. Note that Assumption~\ref{hyp:Feller} is not modified by the rescaling (up to
  appropriate changes of the constants $a$, $B_a$, $C_a$ and $D_a$).

  We divide the proof into five steps, respectively devoted to the
  construction of a function $\varphi$ satisfying the first inequality in~\eqref{eq:true-Lyap-1}, of a function $V$ satisfying
  the second inequality in~\eqref{eq:true-Lyap-1}, to the
  proof of~\eqref{eq:survie<V}, to the proof of a local Harnack
  inequality (needed to check Assumption~\ref{hyp:Doeblin-local}), and the proofs of~Assumption~\ref{hyp:Doeblin-local} and 
  of~Assumption~\ref{hyp:moments-expo}.

  \medskip\noindent\emph{Step 1: construction of a function $\varphi$ satisfying the first inequality in~\eqref{eq:true-Lyap-1}.}

  Recall the definition of the constants $a>0$ and $B_a>a$ from
  Assumption~\ref{hyp:Feller}. We use the following lemma, whose proof
  is left to the reader (see Figure~\ref{fig1} for a typical graph of
  $h_\beta$).

  \begin{lem}
    \label{lem:construction-fonction}
    There exists $M>0$ such that, for all $\beta\geq M$, there exists a function $h_\beta:\RR_+\rightarrow\RR_+$ twice continuously differentiable on $(0,+\infty)$
    such that
    $$
    h_\beta(x)=
    \begin{cases}
      4x^2/a^2 & \text{if }x\in[0,a/2], \\
      B_a^\beta (2x)^{-\beta} & \text{if }x\geq B_a,
    \end{cases}
    $$
    $h_\beta(x)\geq 1$ for all $x\in[a/2,a]$, $h_\beta$ is nonincreasing and convex on $[a,+\infty)$,
    $$
    M':=\sup_{\beta\geq M}\sup_{x\in[a/2,a]}|h'_\beta(x)|<+\infty\quad\text{and}\quad
    M'':=\sup_{\beta\geq M}\sup_{x\in[a/2,a]}|h''_\beta(x)|<+\infty.
    $$
  \end{lem}

  \begin{figure}
    \centering
    \includegraphics[width=11cm]{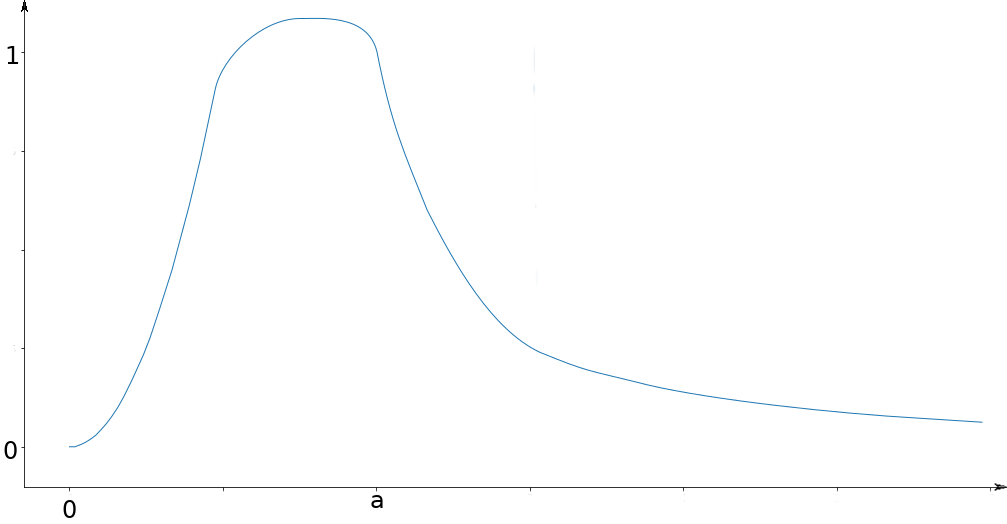}
    \caption{\label{fig1} A typical graph of $h_\beta$.}
  \end{figure}

  We set $\beta=M+(2\vee aM')/C_a+1$ and
  $$
  \varphi(x)=\prod_{i=1}^d h_\beta(x_i),\quad\forall x\in\RR_+^d.
  $$
We have
  \begin{align*}
    \frac{L\varphi(x)}{\varphi(x)} & =\sum_{i=1}^d\frac{x_ih'_\beta(x_i)r_i(x)+x_ih''_\beta(x)}{h_\beta(x_i)}.
  \end{align*}
  Now, it follows from the properties of $h_\beta$
  and Assumptions~\ref{hyp:Feller}
  that, for all $x\in\RR_+$ and all $1\leq i\leq d$,
  \begin{equation*}
    \frac{x_ih'_\beta(x_i)r_i(x)+x_ih''_\beta(x)}{h_\beta(x_i)}\geq
    \begin{cases}
      0 & \text{if }x_i\geq a, \\
      -\beta r_i(x) & \text{if }x_i\geq B_a,\\
      2r_i(x)+\frac{2}{x_i} & \text{if }x_i\leq a/2, \\
      aM'(r_i(x)-a^\eta)-a^{1+\eta}M'-aM''& \text{if }a/2\leq x_i\leq a,
    \end{cases}
  \end{equation*}
  where we used in the last inequality the fact that $r_i(x)-a^\eta\leq 0$ for all $x$. Using once again this property, we deduce
  that, for some constant $B$ independent of $\beta\geq M$,
  \begin{equation}
    \label{eq:borne}
    \frac{x_ih'_\beta(x_i)r_i(x)+x_ih''_\beta(x)}{h_\beta(x_i)}\geq
    \begin{cases}
      0 & \text{if }x_i\geq a, \\
      -\beta r_i(x) & \text{if }x_i\geq B_a,\\
      (2\vee aM')r_i(x)+\frac{2}{x_i}-B & \text{if }x_i\leq a.
    \end{cases}
  \end{equation}
  Hence, for all $x\in\RR_+^d$,
  \begin{align*}
    \frac{L\varphi(x)}{\varphi(x)} & \geq-\beta \sum_{i=1}^d\mathbbm{1}_{x_i\geq B_a}r_i(x)+\sum_{i=1}^d\mathbbm{1}_{x_i\leq a}\left((2\vee aM') r_i(x)+\frac{2}{x_i}\right) - dB.
  \end{align*}
This and Assumption~\eqref{eq:hyp-Feller-3} imply that
\begin{align*}
  \frac{L\varphi(x)}{\varphi(x)} & \geq -\sum_{i=1}^d\mathbbm{1}_{x_i\geq B_a}r_i(x)+2\sum_{i=1}^d\mathbbm{1}_{x_i\leq
    a}\frac{1}{x_i}-dB-(2\vee aM')D_a \\
  &\geq \sum_{i=1}^d\left(x_i^\eta+\frac{2}{x_i}\right)-B',
\end{align*}
for some constant $B'$, where we used Assumption~\eqref{eq:hyp-Feller-1} in the last inequality.

Hence, there exist $n\geq 1$ and a constant $C>0$ such that
\begin{align*}
L\varphi(x)\geq -C\11_{x\in O_n},
\end{align*}
since $O_n=\{x\in[\frac{1}{n+1},+\infty)^d,\,\sum x_i\leq n+1\}$ (we equip $\mathbb{R}_+^d$ with the $L^1$ distance).
This ends the proof that $\varphi$ satisfies the first inequality in~\eqref{eq:true-Lyap-1}.

\medskip\noindent\emph{Step 2: construction of a function $V$
  satisfying~\eqref{eq:true-Lyap-1} and verification that
  $(V,\varphi)$ is a couple of admissible functions.}

  For $V$, we define
  $$
  V(x)=\prod_{i=1}^d g(x_i),\quad\forall x\in\RR_+^d,
  $$
  where the function $g:\RR_+\rightarrow\RR_+$ is twice continuously differentiable on $(0,+\infty)$, increasing concave and such
  that
  $$
  g(x)=
  \begin{cases}
    x^{\gamma} & \text{if }x\leq 1 \\
    \delta-x^{-\eta/2} & \text{if }x\geq 2,
  \end{cases}
  $$
  for some constants $\gamma<1$ and $\delta>0$ and where $\eta$ is
  defined in Assumption~\eqref{eq:hyp-Feller-1}. Since $g'(1)=\gamma$
  and $g'(2)=\eta 2^{-2-\eta/2}$, it is possible to find $\delta>0$
  such that such a function $g$ exists as soon as
  $\eta 2^{-2-\eta/2}<\gamma$. Hence, we shall assume that
  $\gamma$ belongs to the non-empty interval $(\eta
  2^{-2-\eta/2},1)$. We have
  \begin{align*}
    \frac{L V(x)}{V(x)} & =\sum_{i=1}^d\frac{x_ig'(x_i)r_i(x)+x_ig''(x_i)}{g(x_i)}
  \end{align*}
  and
  \begin{align*}
    \frac{x_ig'(x_i)r_i(x)+x_ig''(x_i)}{g(x_i)}\leq
    \begin{cases}
\displaystyle      \gamma r_i(x)-\frac{\gamma(1-\gamma)}{x_i} & \text{if }x_i\leq 1, \\
\displaystyle      2a^\eta\sup_{1\leq x\leq 2}g'(x) & \text{if }1\leq x_i\leq 2, \\
\displaystyle      \frac{\eta r_i(x) x_i^{-\eta/2}}{2(\delta-x_i^{-\eta/2})} & \text{if }x_i\geq 2.
    \end{cases}
  \end{align*}
We deduce from Assumptions~\eqref{eq:hyp-Feller-1} that there exist constants $B',B''>0$ such that
  \begin{align}
  \label{eq:LV-bounded}
    \frac{x_ig'(x_i)r_i(x)+x_ig''(x)}{g(x_i)}\leq B'-
    \begin{cases}
      \frac{\gamma(1-\gamma)}{x_i} & \text{if }x_i\leq 1, \\
      0 & \text{if }1\leq x_i\leq 2, \\
      B''x_i^{\eta/2} & \text{if }x_i\geq 2.
    \end{cases}
  \end{align}
  Thus, since $h_{\beta}(x_i)\geq D_{\beta} \left(x_i^2\wedge x_i^{-\beta}\right)$ for some constant $D_{\beta}>0$ and since $g(x_i)\leq \delta$,
  \begin{align*}
    \frac{L V(x)}{V(x)}+\frac{V(x)^\varepsilon}{\varphi(x)^\varepsilon} & \leq
    B'd-\gamma(1-\gamma)\sum_{i=1}^d\frac{\mathbbm{1}_{x_i\leq 1}}{x_i}-B'' \sum_{i=1}^d\mathbbm{1}_{x_i\geq 2}x_i^{\eta/2} \\ &
    +\left(\frac{\delta}{D_{\beta}}\right)^{d\varepsilon}\prod_{i=1}^d
    \left(x_i^{\varepsilon\beta}\vee x_i^{-2\varepsilon}\right) \\
    & \leq B'd+\gamma(1-\gamma)+B'' 2^{\eta/2}-\gamma(1-\gamma)\left(\inf_i x_i\right)^{-1}-B''\left(\sup_i x_i\right)^{\eta/2} \\ &
    +\left(\frac{\delta}{D_{\beta}}\right)^{d\varepsilon}\left[\left(\sup_i x_i\right)^{\beta d\varepsilon}+\left(\inf_i
        x_i\right)^{-2d\varepsilon}\right].
  \end{align*}
  Therefore, choosing $\varepsilon>0$ such that $d(1+\alpha)\varepsilon<1$ and $\beta d\varepsilon<\eta/2$,
  $LV(x)+V(x)^{1+\varepsilon}/\varphi(x)^\varepsilon\leq 0$ for all $x\in\RR_+^d$ such that $\inf_i x_i$ is small enough or $\sup_i
  x_i$ is big enough. Since $LV$ is bounded from above by~\eqref{eq:LV-bounded} and since $V$ and $\varphi$ are positive continuous on any compact subset of $(0,+\infty)^d$, we have proved Condition~\eqref{eq:true-Lyap-1}.

  We can now check that $(V,\varphi)$ is a couple of admissible functions. First, $V$ and $\varphi$ are both bounded, positive on
  $(0,+\infty)$ and vanishing on $\d$. They both belong to the domain of the weakened infinitesimal generator of $X$. Since the
  function $g/h_\beta$ is positive continuous on $(0,+\infty)$ and
  $$
  \frac{g(x)}{h_\beta(x)}=
  \begin{cases}
    x^{\gamma-2} & \text{if }x\leq 1\wedge a/2, \\
    (\delta-x^{-\eta/2})(2x)^\beta/B_a^\beta & \text{if }x\geq B_a\vee 2,
  \end{cases}
  $$
  we deduce that~\eqref{eq:phi=o(V)} holds true. Condition~\eqref{eq:weak-weak-continuity-V} is also clear since $X_{T_n}\rightarrow
  X_{\tau_\d}$ almost surely. Finally, since $LV$ is bounded from above and $L\varphi$ is bounded from below, we have proved that
  $(V,\varphi)$ is a couple of admissible functions.

  \medskip\noindent\emph{Step 3: proof of~\eqref{eq:survie<V}.}

  Using the upper bound $X^i_t\leq\bar{X}^i_t$ for all $t\geq 0$ and $1\leq i\leq d$, where $\bar{X}^i$ is solution to the
  SDE~\eqref{eq:upper-SDE-2}, and noting that the processes $(\bar{X}^i)_{1\leq i\leq d}$ are independent, we have for all
  $x\in\RR_+^d$ and all $t_2>0$,
  \begin{align*}
    \PP_x(t_2<\tau_\d)\leq \prod_{i=1}^d\PP_{x_i}(\bar{X}^i_{t_2}>0).
  \end{align*}
  Now, there exist constants $D$ and $D'$ such that 
  \begin{equation}
    \label{eq:step3}
    \PP_{x_i}(\bar{X}^i_{t_2}>0)\leq (Dx_i)\wedge 1\leq D'g(x_i)\quad\text{for all }x_i>0.   
  \end{equation}
  To prove this, we can consider a scale function $s$ of the diffusion $\bar{X}^i$ such that $s(0)=0$ and $s(x)>0$ for $x>0$. Using
  the expression of the scale function and the speed measure (see e.g.~\cite[V.52]{rogers-williams-00b}), one easily checks that
  $s(x_i)\sim \alpha x_i$ when $x_i\rightarrow 0$ for some $\alpha\neq 0$ and that Proposition~4.9 of~\cite{ChampagnatVillemonais2015} is satisfied, so that
  $\PP_{x_i}(\bar{X}^i_{t_2}>0)\leq M s(x_i)$ for some $M>0$. Since $s(x_i)\sim \alpha x_i$ when $x_i\rightarrow 0$,~\eqref{eq:step3} is
  proved and hence~\eqref{eq:survie<V} holds true.

\medskip\noindent \emph{Step 4: Harnack inequality for $u$.} 
Consider a bounded nonnegative measurable function $f$ and define the
application $u:(t,x)\in\R_+\times (0,+\infty)^d\mapsto \E_x(f(X_t)\mathbbm{1}_{t<\tau_\d})$. Our aim is to prove that, for all $m\geq 1$, there exist two constants
$N_m>0$ and $\delta_m>0$, which do not depend on $f$, such that
\begin{align}
\label{eq:Harnack-u}
u(\delta_m^2,x)\leq N_m u(2\delta_m^2,y),\text{ for all $x,y\in O_m$ such that }|x-y|\leq \delta_m/2.
\end{align}


We fix $m\geq 1$ and we will omit the indices $m$ for the constants $\delta$ and $N$ for the rest of Step~4. Let $K$ be a compact set
with $C^\infty$ boundary such that $O_m\subset K\subset (0,+\infty)^d$ and such that $d(O_m,\d K)>0$. We set $\delta=d(O_m,\d K)/3$ and
$\tau=\inf\{t\geq 0: X_t\in\partial K\}$. Let $x$ and $y$ be fixed in $K$ such that $|x-y|\leq\delta/2$. We define $\mu_{x}$ and
$\mu_y$ as the joint law of $(\delta^2-\tau\wedge \delta^2,X_{\tau\wedge\delta^2})$ starting from $X_0=x$ and
$(2\delta^2-\tau\wedge(2\delta^2),X_{\tau\wedge(2\delta^2)})$ starting from $X_0=y$, respectively. It follows from Lusin's theorem
(see e.g.~\cite[Thm.\,7.5.2]{Dudley2002}) that there exists a sequence $(h_n)_{n\geq 1}$ of bounded $C^\infty$ functions from
$([0,\infty[\times \d K)\cup (\{0\}\times K)$ to $\R_+$ such that $(h_n)_{n\geq 1}$ converges bounded pointwisely toward
$u$, $(\mu_x+\mu_y)$-almost everywhere.
 For all $n\geq 1$, we let
 $u_n:[0,+\infty[\times K\rightarrow \R$ be the solution to the linear parabolic equation 
\begin{align*}
\begin{cases}
\displaystyle \d_t u_n = \sum_{i=1}^d x_i\frac{\d^2 u_n}{\d x_i^2}+ x_ir_i(x)\frac{\d u_n}{\d x_i},\ \forall (t,x)\in (0,\infty)\times \text{int}(K),\\
\displaystyle u_n(t,x)=h_n(t,x),\ \forall (t,x)\in ([0,\infty[\times \d K)\cup (\{0\}\times K).
\end{cases}
\end{align*}
By~\cite[Thm~5.1.15]{Lunardi1995}, for all $n\geq 1$, $u_n$ is of regularity $C^{1,2}$. In particular, applying It\^o's formula to
$s\mapsto u_n(\delta^2-s,X_s)$ at time $\tau\wedge \delta^2$ and taking the expectation, one deduces that,
\begin{align*}
u_n(\delta^2,x)&=\E_x\left[u_n(\delta^2-\tau\wedge \delta^2,X_{\tau\wedge \delta^2})\right]=\mu_x(h_n)
\end{align*}
By Lebesgue's theorem, the last quantity converges when $n\rightarrow+\infty$ to
\begin{align*}
\mu_x(u) & =\E_x\left(\11_{\tau\leq \delta^2}\, u(\delta^2-\tau,X_{\tau})\right)+\E_x\left(\11_{\tau> \delta^2}\, u(0,X_{\delta^2})\right)\\
&=\E_x\left(\11_{\tau\leq \delta^2}\, \E_{X_{\tau}}(f(X_{\delta^2-\tau})\11_{\delta^2-\tau<\tau_\d})\right)+\E_x\left(\11_{\tau> \delta^2}f(X_{\delta^2})\right)\\
&=\E_x(f(X_{\delta^2})\11_{\delta^2<\tau_\d})=u(\delta^2,x),  
\end{align*}
where we used the strong Markov property at time $\tau$. Similarly, $u_n(2\delta^2,y)$ converges to $u(2\delta^2,y)$.

Using the Harnack inequality provided by~\cite[Theorem~1.1]{KrylovSafonov1980} (with $\theta=2$ and $R=\delta$), we deduce that there
exists a constant $N>0$ which does not depend on $f$, $x,y\in K$ such that $|x-y|\leq\delta/2$ nor on $n$ such that
\begin{align*}
u_n(\delta^2,x)\leq N u_n(2\delta^2,y). 
\end{align*}
Hence, we deduce that
\begin{align}
\label{eq:Harnack-un}
u(\delta^2,x)\leq N u(2\delta^2,y),\text{ for all $x,y\in K$ such that }|x-y|\leq \delta/2
\end{align}
and~\eqref{eq:Harnack-u} is proved.

\medskip\noindent\emph{Step~5 : proof that
  Assumptions~\ref{hyp:Doeblin-local} and~\ref{hyp:moments-expo} are
  satisfied.} Fix $x_1\in O_1$ and let $\nu$ denote the conditional
law $\P_{x_1}(X_{\delta_1^2}\in \cdot\mid \delta_1^2<\tau_\d)$. Then
the Harnack inequality~\eqref{eq:Harnack-u} entails that, for all
$x\in O_1$ such that $|x-x_1|\leq \delta_1/2$ and all measurable nonnegative bounded $f$
on $(0,+\infty)^d$,
\[
\E_x\left[f(X_{2\delta^2_1})\11_{2\delta^2_1<\tau_\d}\right]\geq\frac{1}{N_1}\E_{x_1}\left[f(X_{\delta^2_1})\11_{\delta^2_1<\tau_\d}\right].
\]
This means that
\begin{align*}
  \P_x(X_{2\delta_1^2}\in \cdot)\geq \frac{\P_{x_1}(\delta_1^2<\tau_\d)}{N_1}\,\nu.
\end{align*}
Now, let $m\geq 1$. Since $O_m$ is bounded, connected and at a positive distance of $\d$, $\P_x(X_1\in O_1\cap B(x_1,\delta_1/2))$ is
uniformly bounded from below in $O_m$ by a positive constant $M_m$. Therefore, Markov's property implies that, for all $x\in O_m$,
\begin{align*}
  \P_x(X_{1+2\delta_1^2}\in \cdot)\geq \frac{\P_{x_1}(\delta_1^2<\tau_\d) \,M_n}{N_1}\,\nu.
\end{align*}
This is the first part of Assumption~\ref{hyp:Doeblin-local}.

The second part of Assumption~\ref{hyp:Doeblin-local} is also a consequence
of~\eqref{eq:Harnack-u}. Indeed, for any fixed $m$ and for all
$t\geq 2\delta_m^2$, this equation applied to
$f(x)=\P_x(t-2\delta_m^2<\tau_\d)$ and the Markov property entail that
\begin{align*}
\P_x(t-\delta_m^2<\tau_\d)\leq N_m \P_y(t<\tau_\d),\text{ for all $x,y\in O_m$ such that }|x-y|\leq \delta_m/2.
\end{align*}
Since $s\mapsto \P_x(s<\tau_\d)$ is non-increasing, we deduce that
\begin{align*}
\P_x(t<\tau_\d)\leq N_m \P_y(t<\tau_\d),\text{ for all $x,y\in O_m$ such that }|x-y|\leq \delta_m/2.
\end{align*}
Since $O_m$ has a finite diameter and is connected, we deduce that there exists $N'_m$ such that, for all $t\geq 2\delta_m^2$,
\begin{align*}
\P_x(t<\tau_\d)\leq N'_m \P_y(t<\tau_\d),\text{ for all $x,y\in O_m$.}
\end{align*}
Now, for $t\leq 2\delta_m^2$, we simply use the fact that $x\mapsto \P_x(2\delta_m^2<\tau_\d)$ is uniformly bounded from below on $O_m$ by a constant $1/N''_m>0$. In particular, 
\begin{align*}
\P_x(t<\tau_\d)\leq 1\leq N''_m \P_y(2\delta_m^2<\tau_\d)\leq N''_m \P_y(t<\tau_\d),\text{ for all $x,y\in O_m$.}
\end{align*}
As a consequence, the second part of
Assumption~\ref{hyp:Doeblin-local} is satisfied.

Assumption~\ref{hyp:moments-expo} is a direct consequence of the domination by solutions to~\eqref{eq:upper-SDE-2}, since these
solutions come down from infinity and hit $0$ in finite time almost surely (cf.\ e.g.~\cite{CCLMMS09}).

\bigskip Finally, we deduce from Steps~1, 2, 3 and~5 that all the assumptions of Theorem~\ref{thm:Lyapunov} are satisfied. This concludes the proof of Theorem~\ref{thm:FellerMain}.
\end{proof}

\bibliographystyle{abbrv}
\bibliography{biblio-denis,biblio-math}

\end{document}